\documentclass[10pt]{article}
\usepackage{amsmath,amssymb,bm,amsthm,array,tikz,enumerate}
\usepackage{url}
\usepackage{lscape}
\usepackage[margin=30truemm]{geometry}
\usepackage{float}
\usepackage{makecell}

\usetikzlibrary{positioning,graphs}

\numberwithin{equation}{section}

\DeclareMathOperator{\Spec}{Spec}
\DeclareMathOperator{\tr}{tr}

\DeclareMathOperator{\Ker}{Ker}

\DeclareMathOperator{\Cay}{Cay}
\DeclareMathOperator{\Aut}{Aut}

\newcommand{\MC}[1]{\mathcal{#1}}
\newcommand{\MB}[1]{\mathbb{#1}}

\newcommand{\BM}[1]{{\bm #1}}

\newtheorem{thm}{Theorem}[section]
\newtheorem{lem}[thm]{Lemma}
\newtheorem{pro}[thm]{Proposition}
\newtheorem{cor}[thm]{Corollary}
\newtheorem{question}[thm]{Question}

\theoremstyle{definition}
\newtheorem{defi}[thm]{Definition}

\allowdisplaybreaks

\newcommand{\G}{\Gamma}

\newcommand{\sikaku}{\mathrel{\square}}

\allowdisplaybreaks

\title
{
Strongly regular and strongly walk-regular graphs that admit perfect state transfer
}

\author{
\hfill
\makebox[0.4\textwidth][c]{
Sho Kubota\thanks{
Department of Mathematics Education,
Aichi University of Education,
1 Hirosawa, Igaya-cho, Kariya, Aichi 448-8542, Japan.
\texttt{skubota@auecc.aichi-edu.ac.jp}}
}
\hfill
\makebox[0.4\textwidth][c]{
Hiroto Sekido\thanks{
Graduate School of Environment Information Sciences, Yokohama National University, Hodogaya,
Yokohama 240-8501 Japan.
\texttt{sekido-hiroto-zk@ynu.jp}}
}
\hfill
\and
\hfill
\makebox[0.4\textwidth][c]{
Harunobu Yata 
}
\hfill
\makebox[0.4\textwidth][c]{
Kiyoto Yoshino\thanks{
Department of Computer Science,
Faculty of Applied Information Science,
Hiroshima Institute of Technology,
Saeki Ward, Hiroshima, 731-5143, Japan.
\texttt{k.yoshino.n9@cc.it-hiroshima.ac.jp}
}
}
\hfill
}
\date{}

\begin{document}
\maketitle
\begin{abstract}
We study perfect state transfer in Grover walks on two important classes of graphs: strongly regular graphs and strongly walk-regular graphs.
The latter class is a generalization of the former.
We first give a complete classification of strongly regular graphs that admit perfect state transfer.
The only such graphs are the complete bipartite graph $K_{2,2}$ and the complete tripartite graph $K_{2,2,2}$.
We then show that, 
if a connected strongly walk-regular graph that is not a strongly regular graph
admits perfect state transfer,
then its spectrum must be of the form $\{[k]^1, [\frac{k}{2}]^{\alpha}, [0]^{\beta}, [-\frac{k}{2}]^{\gamma}\}$,
and we enumerate all feasible spectra of this form up to $k=20$ with the help of a computer.
These results are obtained using techniques from algebraic number theory and spectral graph theory,
particularly through the analysis of eigenvalues and eigenprojections of a normalized adjacency matrix.
While the setting is in quantum walks,
the core discussion is developed entirely within the framework of spectral graph theory.
\vspace{8pt} 
\end{abstract}

\section{Introduction}

Quantum walks are a powerful tool in quantum computing and quantum information processing, with potential applications in quantum algorithms~\cite{A}, quantum simulations~\cite{NSCR},
and quantum cryptography~\cite{VRMPS}. 
They have also been studied from a mathematical perspective, as they exhibit interesting properties not typically found in classical random walks.
Here, by {\it classical random walks} we mean Markov chains whose transition probability matrix is given by $D^{-1}A$, where $D$ and $A$ denote the degree matrix and the adjacency matrix of a graph, respectively.
In this sense, phenomena specific to quantum walks include localization and periodicity, both of which can occur under certain conditions and in particular graph structures.

Perfect state transfer, the main topic of this paper,
is also a phenomenon specific to quantum walks, except in trivial cases.
Roughly speaking, perfect state transfer refers to the phenomenon where a quantum state initially localized at one vertex is transferred perfectly to another vertex at a specific time.
In the context of quantum information theory,
it guarantees faithful transfers between specific quantum states,
and thus
has important applications in quantum communication and quantum information processing~\cite{CDEJ}.

The quantum walk model used in this paper is Grover walk,
which is a typical discrete-time quantum walk.
Previous studies for perfect state transfer in Grover walks are summarized in Table~\ref{0228-1}.
For its counterpart in continuous-time quantum walks,
see the survey by Godsil~\cite{G2012} or the unpublished note by Coutinho and Godsil~\cite{CG}.

\begin{table}[h]
  \centering
  \caption{Previous studies on perfect state transfer in Grover walks} \label{0228-1}
  \begin{tabular}{|c|c|}
\hline
Graphs & Ref. \\
\hline
\hline
Diamond chains & \cite{KT} \\ \hline
$\overline{K_2} + \overline{K_n}$,
$\overline{K_2} + C_n$ & \cite{BPAK} \\ \hline
$\Cay(\MB{Z}_{2 \ell}, \{\pm a, \pm b\})$ with $a+b=\ell$ & \cite{Z2019} \\ \hline
$K_{m, m, \dots, m}$ & \cite{KS}  \\ \hline
Circulant graphs with valency up to $4$ & \cite{KY}  \\ \hline
Unitary Cayley graphs & \cite{BB2024-1} \\ \hline
Quadratic unitary Cayley graphs & \cite{BB2024-2} \\ \hline
\makecell{Unitary and quadratic unitary Cayley graphs \\ over finite commutative rings} & \cite{BB2025} \\ \hline

\end{tabular}
\end{table}

See later sections for detailed definitions of terminology.
Here, we first give an overview of our main theorems.
We study strongly regular and strongly walk-regular graphs that admit perfect state transfer.
For strongly regular graphs,
we completely classify the graphs that admit perfect state transfer.
This is the first main theorem.

\begin{thm}
A connected strongly regular graph $\G$ admits perfect state transfer if and only if it is isomorphic to $K_{2,2}$ or $K_{2,2,2}$.
\end{thm}

It should be emphasized that this statement was essentially first discovered by Guo and Schmeits~\cite{GS}.
They studied a weaker notion of perfect state transfer,
called {\it peak state transfer},
and classified the strongly regular graphs that admit peak state transfer.
Based on their result,
our theorem immediately follows from another previous study.
Nevertheless, the proof we present here is independent and relies on algebraic integers.

Furthermore, we investigate perfect state transfer in strongly walk-regular graphs, which is one of the classes that include strongly regular graphs.
\emph{Strongly walk-regular graphs} were introduced by Van Dam and Omidi~\cite{vDO} with the property that the number of $\ell$-walks between two vertices $x,y$ depends only on whether they are equal, adjacent, or non-adjacent.
See Section~\ref{S4} for more detailed definitions.
Strongly walk-regular graphs fall into one of the following classes:
empty graphs,
disjoint unions of complete graphs with the same order,
connected strongly regular graphs,
disjoint unions of complete bipartite graphs of the same size and isolated vertices,
or connected regular graphs with four distinct eigenvalues.
We call graphs in the last class \emph{genuine}.
The second main theorem provides the conditions that the eigenvalues of genuine strongly walk-regular graphs admitting perfect state transfer must satisfy.

\begin{thm}
Let $\G$ be a genuine strongly walk-regular graph with distinct eigenvalues $k > \theta_1 > \theta_2 > \theta_3$.
If $\G$ admits perfect state transfer, then $(\theta_1, \theta_2, \theta_3) = (\frac{k}{2}, 0, -\frac{k}{2})$. 
\end{thm}

This theorem imposes a strong constraint on the spectra of genuine strongly walk-regular graphs that admit perfect state transfer. 
In Section~\ref{S5},
we list all feasible graphs (up to $k = 20$) whose spectrum is of the form
\[ \left \{[k]^1, \left[ \frac{k}{2} \right]^{\alpha}, [0]^{\beta}, \left[-\frac{k}{2} \right]^{\gamma} \right\}. \]
Readers interested in the tables first may refer to Table~\ref{t1} and Table~\ref{t2}.

While our main focus is on perfect state transfer,
we also consider {\it periodicity} as a secondary topic.
We note that, in this paper,
periodicity does not refer to perfect state transfer between the same vertices.
The details will be discussed in Section~\ref{S4},
but since perfect state transfer and periodicity share a common necessary condition,
these phenomena can be analyzed simultaneously.

\section{Preliminaries}

See \cite{GR} for basic terminologies related to graphs.
Let $\G =(V, E)$ be a graph with the vertex set $V$ and edge set $E$.
Throughout this paper, we assume that graphs are simple and finite,
i.e., $|V| < \infty$ and $E \subset \{\{x,y\} \subset V \mid x \neq y\}$.
Define $\MC{A} = \MC{A}(\G)=\{ (x, y), (y, x) \mid \{x, y\} \in E \}$,
which is the set of the \emph{symmetric arcs} of $\G$.
The origin $x$ and terminus $y$ of $a=(x, y) \in \MC{A}$ are denoted by $o(a)$ and $t(a)$, respectively.
We write the inverse arc of $a$ as $a^{-1}$.

\subsection{Grover walks and related matrices}
We define several matrices on Grover walks.
Note that Grover walks are also referred to as arc-reversal walks or arc-reversal Grover walks,
and they are known as a special case of bipartite walks or two-reflection walks.
Let $\G = (V, E)$ be a graph, and set $\MC{A} = \MC{A}(\Gamma)$.
The \emph{boundary matrix} $d = d(\G) \in \MB{C}^{V \times \MC{A}}$ is defined by
$d_{x,a} = \frac{1}{\sqrt{\deg x}} \delta_{x, t(a)}$,
where $\delta_{a,b}$ is the Kronecker delta.
By directly calculating the entries based on the definition of matrix multiplication, it follows that
$dd^* = I$,
where $I$ is the identity matrix.
The \emph{shift matrix} $R = R(\G) \in \MB{C}^{\MC{A} \times \MC{A}}$
is defined by $R_{a, b} = \delta_{a,b^{-1}}$.
Clearly,
$R^2 = I$.
Define the \emph{time evolution matrix} $U = U(\G) \in \MB{C}^{\MC{A} \times \MC{A}}$
by $U = R(2d^*d-I)$.
The entries of $U$ are given by
\[ U_{a,b} = \frac{2}{\deg_{\G} t(b)} \delta_{o(a), t(b)} - \delta_{a,b^{-1}}, \]
as shown in \cite[Lemma~5.1]{KSeYa}.
The quantum walk defined by $U$ is called the \emph{Grover walk} on $\Gamma$.
The spectral analysis of $U$ can be reduced to the spectral analysis of the normalized adjacency matrix.
The \emph{adjacency matrix} $A=A(\G) \in \MB{C}^{V \times V}$ of $\G$ is defined by
\[ A_{x,y} = \begin{cases}
1 \qquad &\text{if $\{ x,y \} \in E$,} \\
0 \qquad &\text{otherwise.}
\end{cases} \]
The \emph{spectrum} of $\G$, denoted by $\Spec(\G)$, is the multiset of eigenvalues of $A(\G)$.
The spectrum of a graph satisfies some basic relations, as shown below.

\begin{lem}[\cite{BH}] \label{0127-1}
Let $\G = (V, E)$ be a graph with spectrum $\Spec(\G) = \{[\lambda_1]^{m_1}, \dots, [\lambda_s]^{m_s}\}$.
Then, the following holds:
\[ \sum_{i=1}^s m_i = |V|, \qquad
\sum_{i=1}^s m_i \lambda_i = 0, \qquad
\sum_{i=1}^s m_i \lambda_i^2 = 2|E|.
\]
\end{lem}

The \emph{discriminant} $P=P(\G) \in \MB{C}^{V \times V}$ is defined by $P = dRd^*$.
If a graph is regular, then the discriminant $P$ and the adjacency matrix are closely related:

\begin{lem}
Let $\G$ be a $k$-regular graph,
and let $A$ and $P$ be the adjacency matrix and the discriminant of $\G$, respectively.
Then we have $P = \frac{1}{k}A$.
\end{lem}

This can be verified by a direct calculation.
A more general claim and its proof can be found in Theorem~3.1 of \cite{KST}.

Since the discriminant $P$ is a normal matrix,
it has a spectral decomposition.
Let $\lambda_1, \dots, \lambda_s$ be the distinct eigenvalues of $P$,
and let $W_i$ be the eigenspace associated to $\lambda_i$.
Let $d_i = \dim W_i$,
and let $u^{(i,1)}, \dots, u^{(i,d_i)}$ be an orthogonal basis of $W_i$.
Then, the matrix
\[ E_i = \sum_{j=1}^{d_i} u^{(i,j)} (u^{(i,j)})^* \]
is called the \emph{eigenprojection} associated to the eigenvalue $\lambda_i$,
and it satisfies 
\[ P = \sum_{i=1}^{s} \lambda_i E_i, \qquad
E_i E_j = \delta_{i,j} E_i, \qquad
E_i^* = E_i, \qquad
\sum_{i=1}^s E_i = I.
\]
See~\cite{M} for the spectral decomposition of matrices.


\subsection{Perfect state transfer}
Let $\G$ be a graph,
and let $U = U(\G)$ be the time evolution matrix.
A vector $\Phi \in \MB{C}^{\MC{A}}$ is a \emph{state} if $\| \Phi \| = 1$.
For distinct states $\Phi$ and $\Psi$,
we say that \emph{perfect state transfer} occurs from $\Phi$ to $\Psi$ at time $\tau \in \MB{Z}_{\geq 1}$
if there exists $\gamma \in \MB{C}$ with norm one such that $U^{\tau}\Phi = \gamma \Psi$.
A state $\Phi$ is said to be \emph{vertex type}
if there exists a vertex $x \in V$ such that $\Phi = d^* e_{x}$,
where $e_{x} \in \MB{C}^{V}$ is the unit vector defined by $(e_{x})_z = \delta_{x,z}$.
See~\cite{KY} for a visual interpretation of vertex type states.
We study perfect state transfer between vertex type states.

\begin{defi}
    For a graph $\Gamma$, we say that $\Gamma$ \emph{admits perfect state transfer} if perfect state transfer occurs for $U(\Gamma)$ between two distinct vertex type states.
\end{defi}

We denote by $\sigma(M)$ the set of distinct eigenvalues of a normal matrix $M$.
For a vector $x$, we define $\Theta_M(x) = \{ \lambda_i \in \sigma(M) \mid E_i x \neq 0 \}$.
This is called the \emph{eigenvalue support} of the vector $x$ with respect to $M$.
Note that $\sigma(M)$ is simply a set,
so the multiplicities of eigenvalues are ignored.
The occurrence of perfect state transfer can be characterized in terms of eigenvalues and eigenprojections.

\begin{thm}[{\cite[Theorem~6.5]{KY}}] \label{pst}
Let $\G = (V, E)$ be a graph.
Let $x,y \in V$ be distinct vertices and $\tau \in \mathbb{Z}_{\geq 1}$.
For each eigenvalue $\lambda$ of $P=P(\G)$,
let $E_{\lambda}$ be the eigenprojection of $P$ associated to $\lambda$.
The following are equivalent.
\begin{enumerate}[\textup{(}A\textup{)}]
	\item Perfect state transfer occurs on $\G$ from $d^*e_x$ to $d^*e_y$ at time $\tau$.\label{pst:A}
	\item 
		$T_{\tau}(P)e_x = e_y$ holds,
        where $T_n(x)$ denotes the Chebyshev polynomial of the first kind of degree $n$.\label{pst:B}
	\item All of the following three conditions are
    satisfied:\label{pst:C}
	\begin{enumerate}[\textup{(}a\textup{)}]
		\item For any $\lambda \in \sigma(P)$, $E_{\lambda}e_x = \pm E_{\lambda}e_y$.\label{pst:C:a}
		\item For any $\lambda \in \Theta_{P}(e_x)$,\label{pst:C:b}
			if $E_{\lambda}e_x = E_{\lambda}e_y$,
			then $\lambda = \cos \frac{j}{\tau}\pi$ for some even $j$.
		\item For any $\lambda \in \Theta_{P}(e_x)$,\label{pst:C:c}
			if $E_{\lambda}e_x = - E_{\lambda}e_y$,
			then $\lambda = \cos \frac{j}{\tau}\pi$ for some odd $j$.
	\end{enumerate}
\end{enumerate}
\end{thm}

We supplement the above theorem.
In~\cite[Theorem~6.5]{KY},
the constant $\gamma$  appears in the statement,
but it was later shown by Guo and Schmeits~\cite[Remark~3.6]{GS} that $\gamma = 1$.

\subsection{Algebraic integers}

In principle,
Theorem~\ref{pst} can be used to determine whether a graph admits perfect state transfer.
However, verifying condition~(C) in practice is somewhat cumbersome.
Our strategy, therefore, is to broadly narrow down the candidates using a necessary condition that is easy to apply.
This necessary condition focuses on algebraic integers.

A complex number $\alpha$ is said to be an \emph{algebraic integer}
if there exists a monic polynomial $p(x) \in \MB{Z}[x]$ such that $p(\alpha) = 0$.
Let $\Omega$ denote the set of algebraic integers.
The set $\Omega$ forms a ring.
In particular, it is closed under both addition and multiplication.
A specific example of an algebraic integer is $\zeta_m = e^{\frac{2\pi i}{m}}$ for a positive integer $m$,
since $(\zeta_m)^m = 1$.
Thus, for any integer $k$, $\zeta_m^k$ is an algebraic integer.
Moreover, for a rational number $2k/m$,
the quantity $2\cos \frac{2k}{m}\pi = \zeta_m^k + \zeta_m^{-k}$ is also an algebraic integer. 
The most basic fact about algebraic integers is that any algebraic integer which is also a rational number must be an integer.
A simple consequence of this fact is the following.

\begin{lem} \label{0403-1}
Let $\theta$ and $k$ be integers.
If $\frac{2\theta}{k} \in \Omega \cap [-2,2]$,
then $\theta \in \{ \pm k, \pm \frac{k}{2}, 0 \}$.
\end{lem}

\begin{proof}
Since $\frac{2\theta}{k}$ is a rational algebraic integer,
it must be an integer.
Moreover, the condition $-2 \leq \frac{2\theta}{k} \leq 2$ implies that $\frac{2\theta}{k} \in \{ \pm 2, \pm 1, 0 \}$.
Hence, we have $\theta \in \{ \pm k, \pm \frac{k}{2}, 0 \}$.
\end{proof}

In addition, we use the following fundamental results.

\begin{lem}[{\cite[Proposition~2.34]{J}}] \label{0129-1}
Let $m > 1$ be a square-free integer.
Then,
\[ \Omega \cap \MB{Q}(\sqrt{m}) =
\begin{cases}
\{ p + q\sqrt{m} \mid p,q \in \MB{Z} \} &\quad \text{if $m \equiv 2,3 \pmod{4}$}, \\
\{ p + \frac{1+\sqrt{m}}{2}q \mid p,q \in \MB{Z} \} &\quad \text{if $m \equiv 1 \pmod{4}$}. \end{cases} \]
\end{lem}

\begin{lem}[{\cite[Theorem~9.13]{J}}] \label{0206-2}
Let $p$ be a prime number.
Then,
\[ \Omega \cap \MB{Q}(\zeta_p)
= \{ c_0 + c_1 \zeta_p + \dots + c_{p-2} \zeta_p^{p-2} \mid c_0, c_1, \dots, c_{p-2} \in \MB{Z} \}. \]
\end{lem}


The following are necessary conditions for perfect state transfer to occur.

\begin{lem}[{\cite[Lemma~9.1]{KY}}] \label{0128-5}
Let $\G$ be a graph with the discriminant $P$,
and let $x,y$ be distinct vertices of $\G$.
If perfect state transfer occurs from $d^*e_x$ to $d^*e_y$,
then $2\lambda$ is an algebraic integer for any $\lambda \in \Theta_{P}(e_x)$.
\end{lem}


\section{Strongly regular graphs that admit perfect state transfer} \label{S3}

For a graph,
the \emph{neighborhood} of a vertex $x$ is the set of vertices adjacent to $x$, and is denoted by $N(x)$.
A graph $\G = (V, E)$ is \emph{strongly regular with parameters} $(n, k, \lambda, \mu)$ if it has $n$ vertices, is neither complete nor edgeless,
and satisfies the following condition for any two vertices $x$ and $y$:
\[ |N(x) \cap N(y)| = \begin{cases}
k \qquad &\text{if $x = y$,} \\
\lambda \qquad &\text{if $\{x,y\} \in E$,} \\
\mu \qquad &\text{if $\{x,y\} \not\in E$ and $x \neq y$.}
\end{cases} \]
It is well known that the equality
$k(k-\lambda-1) = (n-1-k)\mu$ holds among the parameters.
For more information on strongly regular graphs,
see \cite{BH, GR}.
The eigenvalues of strongly regular graphs are determined by their parameters.
The multiplicities are not so important in this study,
so we focus only on the eigenvalues,
which are described by the following theorem.

\begin{thm}[{\cite[Section~10]{GR}}] \label{0529-1}
Let $\G$ be a connected strongly regular graph with parameters $(n, k, \lambda, \mu)$.
Then, we have
$\sigma(\G) = \{ k, \theta_+, \theta_- \}$,
where
\[ \theta_{\pm} = \frac{\lambda - \mu \pm \sqrt{(\lambda - \mu)^2 + 4(k - \mu)}}{2}. \]
In particular, connected strongly regular graphs have three distinct  eigenvalues,
and satisfy
\begin{align}
\theta_+ + \theta_- &= \lambda - \mu, \label{1209-1} \\
\theta_+\theta_- &= \mu - k. \label{1209-2}
\end{align}
\end{thm}

Note that a disconnected strongly regular graph is a disjoint union of complete graphs \cite[Lemma~10.1.1]{GR}, and therefore it has only two distinct eigenvalues.

A list of parameters for strongly regular graphs can be found on Brouwer's website~\cite{B}.
Most strongly regular graphs have integer eigenvalues,
but some also have irrational eigenvalues.
Such strongly regular graphs are subject to a strict restriction on their parameters, where the four different parameters can be expressed in terms of a single parameter, as follows.

\begin{thm}[{\cite[Theorem~9.1.3]{BH}}] \label{1205-2}
Let $\G$ be a strongly regular graph with parameters $(n, k, \lambda, \mu)$,
and let $D = (\lambda - \mu)^2 + 4(k - \mu)$.
If $\sqrt{D} \not\in \MB{Q}$,
then $(n, k, \lambda, \mu) = (4\mu+1, 2\mu, \mu-1, \mu)$.
\end{thm}

\begin{lem} \label{1225-2}
Let $\G = (V, E)$ be a $k$-regular graph that is not complete,
and let $P = P(\G)$.
Then, we have $k \in \Theta_{P}(e_x)$ and $|\Theta_{P}(e_x)| \geq 3$ for any vertex $x \in V(\G)$.
\end{lem}
\begin{proof}
Let $A$ be the adjacency matrix of $\G$.
Since $P = \frac{1}{k}A$,
we have $|\Theta_{A}(e_x)| = |\Theta_{P}(e_x)|$.
Let the distinct eigenvalues of $A$ be $\theta_0 = k, \theta_1, \dots, \theta_s$,
with the corresponding eigenprojections $E_0, E_1, \dots, E_s$.
First, since $E_0 = \frac{1}{n}J$,
it follows that $E_0 e_x = \frac{1}{n}\BM{1} \neq 0$.
Thus, we have $k \in \Theta_{A}(e_x)$.
Moreover, since $E_1e_x + \cdots + E_se_x = e_x - E_0e_x = e_x - \frac{1}{n}\BM{1} \neq 0$,
there exists some $j \neq 0$ such that $E_je_x \neq 0$.
Thus, we have $|\Theta_{A}(e_x)| \geq 2$.
To derive a contradiction,
we suppose that $|\Theta_{A}(e_x)| = 2$.
We may assume that $\Theta_{A}(e_x) = \{k, \theta_1\}$ without loss of generality.
We have
\begin{align*}
\Ker(A - \theta_1 I) &\ni E_1 e_x \\
&= (I - E_0 - E_2 - \cdots - E_s)e_x \\
&= e_x - \frac{1}{n}\BM{1}.
\end{align*}
This implies that the vector $e_x - \frac{1}{n}\BM{1}$ is an eigenvector of $A$ corresponding to the eigenvalue $\theta_1$.
We have $A(e_x - \frac{1}{n}\BM{1}) = \theta_1(e_x - \frac{1}{n}\BM{1})$, i.e.,
\begin{equation} \label{1205-1}
(A - \theta_1 I)e_x = \frac{k - \theta_1}{n} \BM{1}.
\end{equation}
Since $\G$ is not complete,
there exists a vertex $y$ that is not adjacent to $x$.
Examining the entry of Equality~\eqref{1205-1} corresponding to the vertex $y$ yields $k=\theta_1$, which is a contradiction.
\end{proof}

\begin{cor} \label{1205-3}
Let $\G = (V, E)$ be a strongly regular graph,
and let $P = P(\G)$.
Then, we have $\Theta_{P}(e_x) = \sigma(P)$
for any vertex $x \in V(\G)$.
\end{cor}
\begin{proof}
From the inclusion $\Theta_{P}(e_x) \subset \sigma(P)$ and Lemma~\ref{1225-2},
we have $3 \leq |\Theta_{P}(e_x)| \leq |\sigma(P)| = 3$.
Hence $\Theta_{P}(e_x) = \sigma(P)$.
\end{proof}

The following is a key lemma for characterizing perfect state transfer in strongly regular graphs.
As will be mentioned later, 
this lemma also provides an alternative proof of a known result on the periodicity of strongly regular graphs, originally established by Higuchi et al~\cite{HKSS2017}.


\begin{lem} \label{1128-1}
Let $\G$ be a connected strongly regular graph with parameters $(n, k, \lambda, \mu)$.
If both $\frac{2\theta_{+}}{k}$ and $\frac{2\theta_{-}}{k}$ are algebraic integers,
then the parameters must satisfy one of the following:
$(n, k, \lambda, \mu) = (2k, k, 0, k)$, $(3\lambda, 2\lambda, \lambda, 2\lambda)$, or $(5,2,0,1)$.
\end{lem}
\begin{proof}
Let $P$ be the discriminant of $\G$.
By Corollary~\ref{1205-3},
we have $\Theta_{P}(e_x) = \sigma(P) = \{1, \frac{\theta_+}{k}, \frac{\theta_-}{k}\}$.
Let $D = (\lambda - \mu)^2 + 4(k - \mu)$.

First, we consider the case where $\sqrt{D} \not\in \MB{Q}$.
By Theorem~\ref{1205-2},
we have $(n, k, \lambda, \mu) = (4\mu+1, 2\mu, \mu-1, \mu)$.
From our assumption, it follows that
\[ \Omega \ni \frac{2\theta_{\pm}}{k} = \frac{-1\pm\sqrt{4\mu+1}}{2\mu}. \]
Since the sum of algebraic integers is again an algebraic integer,
we obtain
\[ \frac{-1 + \sqrt{4\mu+1}}{2\mu} + \frac{-1 - \sqrt{4\mu+1}}{2\mu} = -\frac{1}{\mu} \in \Omega. \]
Recall that 
$\Omega \cap \MB{Q} = \MB{Z}$.
Hence, $-\frac{1}{\mu}$ must be an integer.
This implies $\mu = 1$,
which leads to $(n, k, \lambda, \mu) = (5,2,0,1)$.

Next, we consider the case where $\sqrt{D} \in \MB{Q}$.
Since the eigenvalues of the adjacency matrix are algebraic integers,
the eigenvalues $\theta_{\pm}$ must be integers.
Moreover, noting that $-k \leq \theta_{\pm} < k$,
we obtain from Lemma~\ref{0403-1} that $\theta_{\pm} \in \{-k, \pm \frac{k}{2}, 0\}$.
From $\theta_+ > \theta_-$,
there are exactly six possible patterns for the pair of numbers $\theta_{\pm}$, as shown below:
\[ (\theta_+, \theta_-) = \left( \frac{k}{2}, 0 \right), \left(\frac{k}{2}, -\frac{k}{2} \right), \left(\frac{k}{2}, -k \right), \left(0, -\frac{k}{2} \right), \left(0, -k \right), \left(-\frac{k}{2}, -k \right).\]
The first and last of these cases can be rejected immediately.
Indeed, the first case fails because the eigenvalues do not sum to zero.
The last case leads to $0 \geq \mu - k = \theta_+ \theta_- = \frac{k^2}{2} > 0$.

Case $(\theta_+, \theta_-) = (\frac{k}{2}, -\frac{k}{2}):$
In this case, we have
$\lambda - \mu = \theta_+ + \theta_-  = \frac{k}{2} + \left( - \frac{k}{2} \right) = 0$ by Equality~\eqref{1209-1}.
Thus,
\begin{equation} \label{1209-4}
\lambda = \mu.
\end{equation}
On the other hand,
$\mu - k = \theta_+ \theta_- = \frac{k}{2} \cdot \left( - \frac{k}{2} \right) = -\frac{k^2}{4}$ by Equality~\eqref{1209-2}.
Thus,
\begin{equation} \label{1209-3}
\mu = k \left( 1 - \frac{k}{4} \right).
\end{equation}
Since $\G$ is connected, the parameter $\mu$ is positive,
so Equality~\eqref{1209-3} leads to $k(1 - \frac{k}{4}) = \mu > 0$, i.e., $k < 4$.
However, assuming $k=3$,
the parameter $\mu$ cannot be an integer and must satisfy $k=2$.
Therefore, Equalities~\eqref{1209-4} and~\eqref{1209-3} lead to $\lambda = \mu = 1$.
Since $k(k-\lambda-1) = (n-1-k)\mu$, we have $n=3$.
However, this shows that the graph $\G$ is isomorphic to $K_3$,
which contradicts the assumption that $\G$ is strongly regular.

Case $(\theta_+, \theta_-) = (\frac{k}{2}, -k):$
Equality~\eqref{1209-2} implies that $\mu - k = \theta_+ \theta_- = - \frac{k^2}{2}$, and hence $\mu = k(1- \frac{k}{2})$.
Since $\G$ is connected, the parameter $\mu$ is positive.
Therefore, $k(1- \frac{k}{2}) = \mu > 0$, i.e., $k < 2$,
which is a contradiction.

Case $(\theta_+, \theta_-) = (0, -\frac{k}{2}):$
Equality~\eqref{1209-2} implies that $\mu - k = \theta_+ \theta_- = 0$, and hence
\begin{equation} \label{1209-5}
\mu = k.
\end{equation}
Equality~\eqref{1209-1} leads to $\lambda - \mu = \theta_+ + \theta_- = -\frac{k}{2}$.
From this and Equality~\eqref{1209-5}, we have $k = 2\lambda$.
Again from Equality~\eqref{1209-5}, we have $\mu = k = 2\lambda$.
Since $k(k-\lambda-1) = (n-1-k)\mu$, we have $n=3\lambda$.
Therefore, $(n, k, \lambda, \mu) = (3\lambda, 2\lambda, \lambda, 2\lambda)$.

Case $(\theta_+, \theta_-) = (0, -k):$
Equality~\eqref{1209-2} implies that $\mu - k = \theta_+ \theta_- = 0$, and hence
\begin{equation} \label{1209-6}
\mu = k.
\end{equation}
Equality~\eqref{1209-1} leads to $\lambda - \mu = \theta_+ + \theta_- = -k$.
From this and Equality~\eqref{1209-6},
we have $\lambda = \mu - k = 0$.
Since $k(k-\lambda-1) = (n-1-k)\mu$, we have $n=2k$.
Therefore, $(n, k, \lambda, \mu) = (2k, k, 0, k)$.
\end{proof}

A graph $\G$ is \emph{periodic} if its time evolution matrix satisfies $U^{\tau} = I$ for some positive integer $\tau$.
We note that, in this paper,
the notion of being periodic does not refer to perfect state transfer between the same vertices.
Higuchi et al.\ have shown that the parameters of periodic strongly regular graphs can only be $(2k, k, 0, k)$, $(3\lambda, 2\lambda, \lambda, 2\lambda)$, and $(5,2,0,1)$, using cyclotomic polynomials~\cite{HKSS2017}.
On the other hand, 
if a $k$-regular graph is periodic,
then $\frac{2\theta}{k} \in \Omega$ must hold for any eigenvalue $\theta$ of the adjacency matrix~\cite[Lemma~3.2]{K}.
Thus, Lemma~\ref{1128-1}, presented earlier,
shows that if a strongly regular graph is periodic,
then the parameters must be one of the three cases.
Conversely, it is straightforward to verify that strongly regular graphs with these parameters are periodic.
Hence, Lemma~\ref{1128-1} essentially provides an alternative proof for the classification of periodic strongly regular graphs.

In general, strongly regular graphs cannot be uniquely determined by their parameters alone.
However, each strongly regular graph with the three sets of parameters mentioned in Lemma~\ref{1128-1} is uniquely determined.
Specifically, the strongly regular graph with parameters $(2k, k, 0, k)$ is the complete bipartite graph $K_{k,k}$,
the one with parameters $(3\lambda, 2\lambda, \lambda, 2\lambda)$ is the complete tripartite graph $K_{\lambda, \lambda, \lambda}$,
and the one with parameters $(5,2,0,1)$ is the cycle graph $C_5$ of length $5$.

\begin{thm}
A connected strongly regular graph $\G$ admits perfect state transfer if and only if it is isomorphic to $K_{2,2}$ or $K_{2,2,2}$.
\end{thm}

\begin{proof}
We suppose that perfect state transfer occurs from $d^*e_x$ in $\G$.
By Lemma~\ref{0128-5},
we have $\frac{2\theta}{k} \in \Omega$ for any $\theta \in \Theta_{A}(e_x)$.
Thus, Lemma~\ref{1128-1} implies that the graph $\G$ is isomorphic to either $K_{m,m}$, $K_{m,m,m}$, or $C_5$.
Whether these graphs admit perfect state transfer has been investigated in previous studies,
which have shown that the odd cycle graph does not admit perfect state transfer~\cite[Theorem~5.1]{KY},
whereas complete multipartite graphs with equally sized partite sets that admit perfect state transfer are only $K_{2,2}$ and $K_{2,2,2}$~\cite[Section~4]{KS}.
\end{proof}

Although in this paper we determine which strongly regular graphs admit perfect state transfer using algebraic integers,
the statement of the result essentially appears in the work of Guo and Schmeits.
In~\cite{GS}, they introduce a concept called {\it peak state transfer} and determine which strongly regular graphs admit the phenomenon.
Peak state transfer is, roughly speaking, perfect state transfer without the condition of strong cospectrality, i.e., condition (C)(a) of Theorem~\ref{pst}.
According to their results,
if a connected strongly regular graph on $n$ vertices admits peak state transfer, then it is isomorphic to either the complete bipartite graph $K_{\frac{n}{2}, \frac{n}{2}}$ or the complete tripartite graph $K_{\frac{n}{3}, \frac{n}{3}, \frac{n}{3}}$~\cite[Theorem~5.3]{GS}.
Given this result,
our results follow immediately from~\cite[Section~4]{KS}.

\section{Strongly walk-regular graphs} \label{S4}

Let $\G$ be a graph with adjacency matrix $A$.
The property that $\G$ is a strongly regular graph can be interpreted as the condition that $A^2$ can be expressed as a linear combination of $I$, $A$, and $J$.
A generalization of the exponent $2$ of $A$ to an arbitrary integer $\ell \geq 2$ is the concept of strongly walk-regular graphs.
These graphs were introduced by Van Dam and Omidi~\cite{vDO}.
In this paper,
the parameters of strongly walk-regular graphs are not our main focus,
so we present a slightly modified definition.

Let $\ell \geq 2$ be a positive integer.
A graph $\G$ with adjacency matrix $A$ is said to be \emph{strongly $\ell$-walk-regular} if $A^{\ell}$ can be expressed as a linear combination of $I$, $A$, and $J$.
Moreover, a graph is said to be simply \emph{strongly walk regular} if it is $\ell$-strongly walk regular for some $\ell$.
Note that a graph can be both strongly $\ell$-walk-regular and strongly $m$-walk-regular for different positive integers $\ell$ and $m$.
Indeed, strongly regular graphs are not only strongly $2$-walk-regular but also strongly $\ell$-walk-regular for any $\ell \geq 2$.

Van Dam and Omidi~\cite{vDO} showed that strongly walk-regular graphs belong to one of the following classes:
empty graphs,
disjoint unions of complete graphs with the same order,
connected strongly regular graphs,
disjoint unions of complete bipartite graphs of the same size and isolated vertices,
or connected regular graphs with four distinct eigenvalues.
Following Zhang et al.~\cite{ZHF},
we refer to graphs in the last class (that is,
connected strongly walk-regular graphs with four distinct eigenvalues) as \emph{genuine strongly walk-regular graphs}.
Recently, Zhang et al.\ ~\cite{ZHF} classified genuine strongly $3$-walk-regular graphs with the least eigenvalue $-2$.
Note that if a graph is genuine strongly $\ell$-walk-regular,
then $\ell$ must be an odd number~\cite[Theorem~3.4]{vDO}.

Examples of strongly walk-regular graphs include the Hamming graph $H(3,3)$ and the complement of the Cartesian product of $K_{m,m}$ and $K_m$,
which is denoted by $\overline{K_{m,m} \sikaku K_m}$,
where $m$ is a positive integer.
More sophisticated constructions have also been established,
utilizing coset graphs of the duals of three-weight codes~\cite{CX, MS, SKKS},
or Cayley graphs derived from plateaued Boolean functions~\cite{RSS}.

In fact,
genuine strongly walk-regular graphs are characterized by their eigenvalues:

\begin{lem}[{\cite[Proposition~4.3]{vDO}}] \label{1231-1}
Let $\G$ be a 
connected $k$-regular graph with distinct eigenvalues $k$, $\theta_1$, $\theta_2$, $\theta_3$,
and let $\ell \geq 3$.
Then $\G$ is strongly $\ell$-walk-regular if and only if
\begin{equation} \label{0103-1}
(\theta_2 - \theta_3)\theta_1^{\ell}
+ (\theta_3 - \theta_1)\theta_2^{\ell}
+ (\theta_1 - \theta_2)\theta_3^{\ell} = 0.
\end{equation}
\end{lem}


When $\ell = 3$,
factoring the left-hand side of Equality~\eqref{0103-1} results in
$-(\theta_1 + \theta_2 + \theta_3)(\theta_1 - \theta_2)(\theta_2 - \theta_3)(\theta_3 - \theta_1)$,
which leads to the following:

\begin{cor}[{\cite[Propostion~4.1]{vDO}}] \label{0109-6}
Let $\G$ be a 
connected regular graph with distinct eigenvalues $k, \theta_1, \theta_2, \theta_3$.
Then $\G$ is strongly $3$-walk-regular if and only if $\theta_1 + \theta_2 + \theta_3 = 0$.
\end{cor}

\begin{cor} \label{0113-3}
Let $\G$ be a genuine strongly walk-regular graph
with distinct eigenvalues $k, \theta_1, \theta_2, \theta_3$.
Then, $\theta_2 = 0$ if and only if $\theta_1 + \theta_3 = 0$.
In this case, $\G$ is a strongly $\ell$-walk-regular graph for every odd $\ell$.
\end{cor}

Using Lemma~\ref{1231-1},
we can verify that the cycle graph $C_6$,
which has the distinct spectrum $\sigma(C_6) = \{\pm 2, \pm 1\}$,
is not strongly walk-regular.
However,
applying Lemma~\ref{1231-1} becomes difficult
when a graph has irrational eigenvalues.
To facilitate later discussions,
we confirm that $C_7$ is not strongly walk-regular either.

\begin{lem} \label{0416-1}
The cycle graph $C_7$ of length $7$ is not strongly walk-regular.
\end{lem}
\begin{proof}
We denote the vertex set as $V(C_7) = \{ x_0, x_1, \dots, x_6 \}$
and the edge set as $E(C_7) = \{ \{x_0, x_1\}$, $\{x_1, x_2\}$, $\dots, \{x_5, x_6\}$, $\{x_6, x_0\} \}$.
Let $A = A(C_7)$,
and define $w_j^{(k)} = (A^{k})_{x_0, x_j}$ for $j \in \{0,1, \dots, 6\}$ and a positive integer $k$.
This is the number of walks of length $k$ from $x_0$ to $x_j$.
It is enough to show that $w_2^{(\ell)} \neq w_3^{(\ell)}$ for any $\ell \geq 2$.
By symmetry of the graph,
we have $w_4^{(k)} = w_3^{(k)}$, $w_5^{(k)} = w_2^{(k)}$,
and $w_6^{(k)} = w_1^{(k)}$.
Clearly,
\begin{align}
w_0^{(k+1)} &= 2w_1^{(k)}, \label{1223-a} \\
w_1^{(k+1)} &= w_0^{(k)} + w_2^{(k)}, \label{1223-b} \\
w_2^{(k+1)} &= w_1^{(k)} + w_3^{(k)}, \label{1223-c} \\
w_3^{(k+1)} &= w_2^{(k)} + w_3^{(k)}. \label{1223-d}
\end{align}
We want to show the following two statements:
for any positive integer $k$,
if $w_0^{(k)} > w_2^{(k)} > w_3^{(k)} > w_1^{(k)}$,
then $w_0^{(k+1)} < w_2^{(k+1)} < w_3^{(k+1)} < w_1^{(k+1)}$;
and for any positive integer $k$,
if $w_0^{(k)} < w_2^{(k)} < w_3^{(k)} < w_1^{(k)}$,
then $w_0^{(k+1)} > w_2^{(k+1)} > w_3^{(k+1)} > w_1^{(k+1)}$.
We start with the first statement.
Suppose that $w_0^{(k)} > w_2^{(k)} > w_3^{(k)} > w_1^{(k)}$ for a positive integer $k$.
From this assumption and Equalities~\eqref{1223-a}--\eqref{1223-d}, we have
\begin{align*}
w_1^{(k+1)} - w_3^{(k+1)} &= w_0^{(k)} + w_2^{(k)} - (w_2^{(k)} + w_3^{(k)}) = w_0^{(k)} - w_3^{(k)} > 0, \\
w_3^{(k+1)} - w_2^{(k+1)} &= w_2^{(k)} + w_3^{(k)} - (w_1^{(k)} + w_3^{(k)}) = w_2^{(k)} - w_1^{(k)} > 0, \\
w_2^{(k+1)} - w_0^{(k+1)} &= w_1^{(k)} + w_3^{(k)} - 2w_1^{(k)} = w_3^{(k)} - w_1^{(k)} > 0,
\end{align*}
which implies $w_0^{(k+1)} > w_2^{(k+1)} > w_3^{(k+1)} > w_1^{(k+1)}$.
The second statement can be proven in the same manner.
Furthermore, we observe that
\[ w_0^{(4)} = 6 > w_2^{(4)} = 4 > w_3^{(4)} = 1 > w_1^{(4)} = 0. \]
This implies that $w_2^{(\ell)} \neq w_3^{(\ell)}$ for any $\ell \geq 4$.
Finally, we verify that $w_2^{(3)} = 0 \neq 1 = w_3^{(3)}$ and $w_2^{(2)} = 1 \neq 0 = w_3^{(2)}$, which completes the proof.
\end{proof}

Since strongly walk-regular graphs have at most four distinct eigenvalues \cite[Theorem~3.4]{vDO},
if the cycle graph $C_n$ of length $n$ is strongly walk-regular,
then it must hold that $n \leq 7$.
From the previous discussion,
$C_6$ and $C_7$ are not strongly walk-regular,
whereas $C_4$ and $C_5$ are strongly regular,
and hence they are strongly walk-regular.
Additionally, $C_3$ is the complete graph,
so it is also strongly walk-regular.
Thus, we can now determine cycle graphs that are strongly walk-regular:

\begin{lem} \label{0101-3}
The cycle graph $C_n$ of length $n$ is strongly walk-regular if and only if $n \in \{3,4,5\}$.
In particular,
there is no genuine strongly walk-regular graph with valency $2$.
\end{lem}

As we saw in Section~\ref{S3},
if the eigenvalue support $\Theta_P(e_x)$ of $e_x$ coincides with the set of distinct eigenvalues $\sigma(P)$ of the discriminant $P$,
then periodicity and perfect state transfer share the common necessary condition that $2\theta$ must be an algebraic integer for any eigenvalue $\theta \in \sigma(P)$.
These logical relationships are briefly summarized in Figure~\ref{0126-1}.
Thus, by initially examining genuine strongly walk-regular graphs
that satisfy the condition
\begin{equation} \label{0314-1}
\text{$2\theta \in \Omega$ for any $\theta \in \sigma(P)$,}
\end{equation}
we can simultaneously address both periodicity and perfect state transfer.
First, we show that the eigenvalue support coincides with the set of distinct eigenvalues for genuine strongly walk-regular graphs, under a reasonable assumption.

\begin{figure}[h]
\centering
\begin{tikzpicture}
  \node (periodicity) at (0, 0) {Periodicity};
  \node (pst) at (6, 0) {Perfect state transfer};
  \node (lemma1) at (-2, -0.75) {\cite[Lemma~3.2]{K}};
  \node (arrow1) at (0, -0.75) {$\Downarrow$};
  \node (lemma2) at (7.5, -0.75) {Lemma~\ref{0128-5}};
  \node (arrow2) at (6, -0.75) {$\Downarrow$};
  \node (cond1) at (0, -1.5) {$2\theta \in \Omega$ for any $\theta \in \sigma(P)$};
  \node (iff) at (3, -1.5) {$\iff$};
  \node (cond2) at (6, -1.5) {$2\theta \in \Omega$ for any $\theta \in \Theta_P(e_x)$};
  \node (note) at (3, -2) {$\uparrow$};
  \node (note) at (3, -2.5) {if $\sigma(P) = \Theta_P(e_x)$};
\end{tikzpicture}
\caption{Logical relationships between periodicity and perfect state transfer} \label{0126-1}
\end{figure}
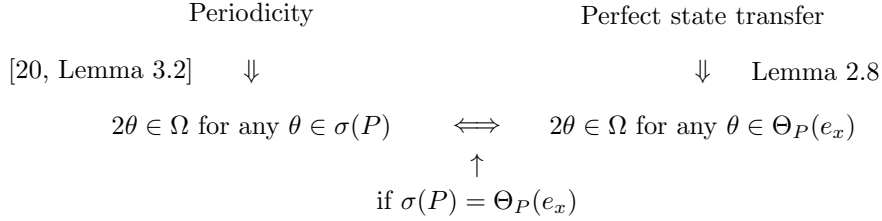

\begin{lem} \label{0124-1}
Let $\G = (V, E)$ be a genuine strongly walk-regular graph with distinct eigenvalues $k$, $\theta_1$, $\theta_2$, $\theta_3$,
and let $P = P(\G)$.
For any vertex $x \in V$,
if $\frac{\theta_1}{k}, \frac{\theta_2}{k} \in \Theta_{P}(e_x)$ and both $\frac{2\theta_1}{k}$ and $\frac{2\theta_2}{k}$ are algebraic integers,
then $\Theta_{P}(e_x) = \sigma(P)$.
\end{lem}

\begin{proof}
We may discuss adjacency matrices.
Let $A$ be the adjacency matrix of $\G$,
and let $E_i$ be the eigenprojection associated to the eigenvalue $\theta_i$ for $i \in \{0,1,2,3\}$,
where $\theta_0 = k$.
Let $n = |V|$.
Suppose, for contradiction, that $|\Theta_{A}(e_x)| = 3$.
Since $E_0 = \frac{1}{n}J$ and $E_3 e_x = 0$,
we have $e_x = (E_0 + E_1 + E_2 + E_3)e_x = \frac{1}{n}{\bm 1} + E_1 e_x + E_2 e_x$.
Thus,
\begin{align*}
e_x - \frac{1}{n} {\bm 1}
&= E_1 e_x + E_2 e_x \\
&\in \Ker(A-\theta_1I) + \Ker(A-\theta_2I) \\
&\subset \Ker(A-\theta_1I)(A-\theta_2I).
\end{align*}
Let $g(u) := (u-\theta_1)(u-\theta_2)$.
Then we have $0 = g(A)(e_x - \frac{1}{n} {\bm 1}) = g(A)e_x - \frac{g(k)}{n}{\bm 1}$, that is,
\begin{equation} \label{1231-2}
A^2 e_x = (\theta_1 + \theta_2)A e_x - \theta_1 \theta_2 e_x + \frac{(k-\theta_1)(k-\theta_2)}{n}{\bm 1}.
\end{equation}
Taking the $x$-entry of Equality~\eqref{1231-2} yields
\begin{equation} \label{1231-3}
k = - \theta_1 \theta_2 + \frac{(k-\theta_1)(k-\theta_2)}{n}.
\end{equation}
Let $y$ be a vertex adjacent to $x$.
Taking the $y$-entry of Equality~\eqref{1231-2} yields
$A^2_{y,x}= \theta_1 + \theta_2 + \frac{(k-\theta_1)(k-\theta_2)}{n}$.
The right-hand side does not depend on the choice of the vertex $y$,
provided that $y$ is adjacent to $x$.
Let $d_1$ denote this value, i.e,
\begin{equation} \label{1231-4}
d_1 = A^2_{y,x} = \theta_1 + \theta_2 + \frac{(k-\theta_1)(k-\theta_2)}{n}.
\end{equation}
Since $\G$ is not complete, there exists a vertex $z$ whose distance from vertex $x$ is $2$.
Taking the $z$-entry of Equality~\eqref{1231-2} yields
$A^2_{z,x} = \frac{(k-\theta_1)(k-\theta_2)}{n}$.
The right-hand side does not depend on the choice of the vertex $z$,
provided that the distance between $z$ and $x$ is $2$.
Let $d_2$ denote this value, i.e,
\begin{equation} \label{1231-5}
d_2 = A^2_{z,x} = \frac{(k-\theta_1)(k-\theta_2)}{n}.
\end{equation}
From Equalities~\eqref{1231-3}--\eqref{1231-5},
we have
\begin{align}
\theta_1 \theta_2 = d_2 - k, \label{0101-4} \\
\theta_1 + \theta_2 = d_1 - d_2. \label{0101-5}
\end{align}
Hence, $g(u) = u^2 - (\theta_1 + \theta_2)u + \theta_1 \theta_2 = u^2 - (d_1 - d_2)u + (d_2 - k)$.

We claim that
\begin{equation} \label{0101-2}
d_1 \leq k - 2.
\end{equation}
To show this inequality,
we count the number of walks of length $2$ starting from the vertex $x$ in two ways.
Let $W := \{ (x, y, z) \mid y,z \in V, x \sim y \sim z \}$.
For any element of $W$,
there are $k$ choices for the vertex $y$,
and for each choice of $y$ there are $k$ ways to choose the vertex $z$.
Thus, we have $|W| = k^2$.
Next, we fix the vertex $z$ and then count it.
\begin{align*}
|W| &= \sum_{z \in V}|\{ (x,y,z) \mid y \in V, x \sim y \sim z \}| \\
&= \sum_{z \in V} A^2_{x,z} \\
&= \sum_{\substack{z \in V \\ d(x,z) \leq 2}} A^2_{x,z} \\
&= \sum_{\substack{z \in V \\ d(x,z) = 0}} A^2_{x,z}
+ \sum_{\substack{z \in V \\ d(x,z) = 1}} A^2_{x,z}
+ \sum_{\substack{z \in V \\ d(x,z) = 2}} A^2_{x,z} \\
&= k + k d_1 + d_2 |N_2(x)|,
\end{align*}
where $N_2(x) = \{ z \in V \mid d(x, z) = 2\}$.
Hence, $k^2 = k + k d_1 + d_2 |N_2(x)|$. 
From this equality, we have $kd_1 = k^2 - k - d_2 |N_2(x)| < k^2 - k$.
This implies that $d_1 \leq k-2$.

Next, let
\[ f(u) := \left( \frac{2}{k} \right)^2 g \left( \frac{ku}{2} \right)
= u^2 - \frac{2(d_1 - d_2)}{k}u + \frac{4(d_2 - k)}{k^2}. \]
Since $g(\theta_1) = g(\theta_2) = 0$,
we have $f(\frac{2\theta_1}{k}) = f(\frac{2\theta_2}{k}) = 0$.
From this and $\frac{2\theta_1}{k}, \frac{2\theta_2}{k} \in [-2, 2)$,
it must hold that
\begin{equation} \label{0101-1}
f(-2) \geq 0.
\end{equation}
Furthermore, since $\frac{2\theta_1}{k}, \frac{2\theta_2}{k} \in \Omega$,
we have
\begin{align*}
\frac{2\theta_1}{k} + \frac{2\theta_2}{k}
&= \frac{2(d_1 - d_2)}{k} \in \Omega \cap \MB{Q} = \MB{Z}, \\
\frac{2\theta_1}{k} \cdot \frac{2\theta_2}{k}
&= \frac{4(d_2 - k)}{k^2} \in \Omega \cap \MB{Q} = \MB{Z}.
\end{align*}
In particular, $f(u) \in \MB{Z}[u]$.
By Inequality~\eqref{0101-1},
\begin{align*}
0 &\leq \left( \frac{k}{2} \right)^2 f(-2) \\
&= k^2 + k(d_1 - d_2) + d_2 - k \\
&\leq k^2 + k(k-2 - d_2) + d_2 - k \tag{by \eqref{0101-2}} \\
&< k^2 + k(k-2 - d_2) + d_2 \\
&= (2k-d_2)(k-1).
\end{align*}
Thus, we obtain $0 < \frac{4d_2}{k} < 8$.
Let $a$ be the constant term of $f(u) \in \MB{Z}[u]$,
that is, $a := \frac{4(d_2 - k)}{k^2}$.
Then we have $\frac{4d_2}{k} = ka + 4 \in \MB{Z} \cap (0,8) = \{1,2,3,4,5,6,7\}$,
and hence $a \in \{0, \pm \frac{1}{k}, \pm \frac{2}{k}, \pm \frac{3}{k} \}$.
Since $a \in \MB{Z}$,
the only possible cases are $a=0$, $(k,a) = (2, \pm 1)$, or $(k,a) = (3, \pm 1)$.
However, Lemma~\ref{0101-3} implies $k \neq 2$,
and the last case $(k,a) = (3, \pm 1)$ leads to a contradiction,
as $\MB{Z} \ni d_2 = \frac{k}{4}(ak+4) \in \{ \frac{3 \cdot 1}{4}, \frac{3 \cdot 7}{4} \}$.
Therefore, $a=0$ must hold.

Finally, we show that $a=0$ also cannot hold.
Since $a=0$, we have $d_2 = k$.
Thus, Equality~\eqref{0101-4} implies $\theta_1 \theta_2 = 0$.
We may assume that $\theta_1 = 0$ without loss of generality.
Then Equality~\eqref{0101-5} leads to $\theta_2 = d_1 - k$.
Moreover, since
\[ \Omega \ni \frac{2\theta_2}{k} = \frac{2(d_1 - k)}{k}
= \frac{2d_1}{k} - 2 \in \MB{Q}, \]
we have $\frac{2d_1}{k} \in \MB{Z}$.
By Inequality~\eqref{0101-2},
$0 \leq \frac{2d_1}{k} \leq \frac{2(k-2)}{k} = 2 - \frac{4}{k} < 2$,
and hence $\frac{2d_1}{k} \in \{0,1\}$.
If $\frac{2d_1}{k} = 0$,
then $d_1 = 0$, which implies $\theta_2 = -k$.
This indicates that $\G$ is bipartite,
and hence its spectrum is symmetric with respect to the origin.
However, since $\theta_1 = 0$,
the remaining eigenvalue must satisfy $\theta_3 = 0$,
which contradicts the assumption that $\G$ has four distinct eigenvalues.
Thus, we have $\frac{2d_1}{k} = 1$.
Then it holds that $d_1 = \frac{k}{2}$ and $\theta_2 = -\frac{k}{2}$.
Furthermore, Equality~\eqref{1231-3} leads to

\begin{equation} \label{0101-6}
n = \frac{3}{2}k.
\end{equation}
Assume that $\G$ is strongly $\ell$-walk-regular.
By Lemma~\ref{1231-1},
we have $\frac{k}{2} \theta_3^{\ell} + (-\frac{k}{2})^{\ell} \theta_3 = 0$.
Since $\theta_3 \not\in \{ 0, -\frac{k}{2}\}$ and $\ell$ is odd,
we obtain $\theta_3 = \frac{k}{2}$.
Now, we can express the spectrum of $\G$ as
$\Spec(\G) = \{ [k]^1, [\frac{k}{2}]^{\alpha}, [0]^{\beta}, [-\frac{k}{2}]^{\gamma} \}$.
From Lemma~\ref{0127-1},
we have
\begin{align*}
1 + \alpha + \beta + \gamma &= n, \\
k + \frac{k}{2} \cdot \alpha + \left( -\frac{k}{2} \right) \gamma &= 0, \\
k^2 + \left( \frac{k}{2} \right)^2 \alpha + \left( -\frac{k}{2} \right)^2 \gamma &= nk.
\end{align*}
Solving this for $\alpha, \beta, \gamma$ gives
\[ \alpha = \frac{2n}{k} - 3, \quad
\beta = n+3 - \frac{4n}{k}, \quad
\gamma = \frac{2n}{k} -1.
\]
In particular, since $\alpha \geq 1$,
we have $n \geq 2k$,
which contradicts~\eqref{0101-6}.
Therefore, $a=0$ is also impossible,
and hence $\Theta_{P}(e_x) = \sigma(P)$.
\end{proof}

\begin{cor} \label{0404-1}
Let $\G = (V, E)$ be a genuine strongly walk-regular graph,
and let $P = P(\G)$.
If $\G$ admits perfect state transfer,
then $\Theta_{P}(e_x) = \sigma(P)$ for any vertex $x \in V$.
\end{cor}

\begin{proof}
Let the eigenvalues of $A(\G)$ be $k, \theta_1, \theta_2, \theta_3$.
By Lemma~\ref{1225-2}, we know that $|\Theta_{P}(e_x)| \geq 3$.
We may assume that $\frac{\theta_1}{k}, \frac{\theta_2}{k} \in \Theta_{P}(e_x)$ without loss of generality.
Since $\G$ admits perfect state transfer,
it follows from Lemma~\ref{0128-5} that $\frac{2\theta_1}{k}$ and $\frac{2\theta_2}{k}$ are algebraic integers.
Therefore, by Lemma~\ref{0124-1},
we conclude that $\Theta_{P}(e_x) = \sigma(P)$.
\end{proof}

In Lemma~\ref{0124-1},
we do not yet know whether the assumptions imposed on the graphs can be removed:

\begin{question}
Let $\G = (V, E)$ be a regular graph with distinct eigenvalues $k, \theta_1, \theta_2, \theta_3$,
and let $P = P(\G)$.
Does the equality $\Theta_{P}(e_x) = \sigma(P)$ hold for any vertex $x \in V$?
\end{question}

The following is a fundamental result by Van Dam on the eigenvalues of regular graphs with four distinct eigenvalues.

\begin{lem}[{\cite[Theorem~2.6]{vD}}] \label{0109-1}
Let $\G$ be a connected $k$-regular graph on $n$ vertices
with spectrum $\{ [k]^1, [\theta_1]^{m_1}, [\theta_2]^{m_2}, [\theta_3]^{m_3} \}$,
and let $m = \frac{n-1}{3}$.
Then one of the following holds:
  \begin{enumerate}[(i)]
  \item $m_1 = m_2 = m_3 = m$ and $k \in \{m, 2m\}$,
  \item $\G$ has precisely two distinct integral eigenvalues, or
  \item $\G$ has four distinct integral eigenvalues.
  \end{enumerate}
Moreover, if $\G$ has precisely two distinct integral eigenvalues,
then the other two have the same multiplicities and are of the form $\frac{1}{2}(a \pm \sqrt{b})$ with $a,b \in \MB{Z}$.
\end{lem}

Therefore, we investigate (i), (ii), and (iii) in Lemma~\ref{0109-1} separately for perfect state transfer and periodicity of genuine strongly walk-regular graphs.
In the following sequence of statements,
we show that genuine strongly walk-regular graphs in cases~(i) and~(ii) do not satisfy the condition~\eqref{0314-1},
whereas graphs in case~(iii) that satisfy the condition~\eqref{0314-1} have their eigenvalues $(\theta_1, \theta_2, \theta_3) = (\frac{k}{2}, 0, -\frac{k}{2})$.

\begin{lem} \label{0109-2}
Let $\G$ be a genuine strongly walk-regular graph with spectrum $\{ [k]^1, [\theta_1]^{m_1}$, $[\theta_2]^{m_2}, [\theta_3]^{m_3} \}$.
Suppose that $\frac{2\theta_i}{k}$ is an algebraic integer for every $i \in \{1,2,3\}$.
Then, the multiplicities $m_1$, $m_2$, and $m_3$ cannot all be equal.
\end{lem}
\begin{proof}
Suppose, for contradiction, that $m_1 = m_2 = m_3 =: m$.
We may assume that $\theta_1 > \theta_2 > \theta_3$ without loss of generality.
The eigenvalues $\lambda$ of $\G$ excluding $k$ satisfy $-k \leq \lambda < k$.
Moreover, 
since $\G$ is not complete,
applying interlacing for a $2$-coclique yields $\theta_1 \geq 0$.
In addition, since the sum of the eigenvalues is zero,
we have
\begin{equation} \label{0103-2}
k + m(\theta_1 + \theta_2 + \theta_3) = 0.
\end{equation}
In particular, we have $\theta_3 < 0$ and $\theta_1 + \theta_2 + \theta_3 < 0$.
Thus, we obtain
\[ 0 \leq\frac{2\theta_1}{k} < 2,
\quad -2 < \frac{2\theta_2}{k} < 2,
\quad -2 \leq \frac{2\theta_3}{k} < 0,\]
and hence $-4 < \frac{2}{k}(\theta_1 + \theta_2 + \theta_3) < 0$.
On the other hand,
Equality~\eqref{0103-2} and the assumption of our lemma
imply that $\frac{2}{k}(\theta_1 + \theta_2 + \theta_3) \in \Omega \cap \MB{Q} = \MB{Z}$.
Therefore, we have $\frac{2}{k}(\theta_1 + \theta_2 + \theta_3) \in \{-1, -2, -3\}$.

If $\frac{2}{k}(\theta_1 + \theta_2 + \theta_3) = -3$,
then by Equality~\eqref{0103-2}, we have $m = \frac{2}{3}$. 
This leads to a contradiction, as the multiplicity $m$ must be an integer.
If $\frac{2}{k}(\theta_1 + \theta_2 + \theta_3) = -2$,
then by Equality~\eqref{0103-2}, we have $m = 1$.
Thus, we obtain $|V(\G)| = 4$.
A connected regular graph with $4$ vertices is either $K_4$ or $C_4$,
but neither of them has four distinct eigenvalues.
Hence, this case also leads to a contradiction.
If $\frac{2}{k}(\theta_1 + \theta_2 + \theta_3) = -1$,
then by Equality~\eqref{0103-2}, we have $m = 2$.
Thus, we obtain $|V(\G)| = 7$.
The handshaking lemma implies that the valency $k$ is even.
Since $\G$ is not complete, we have $k \neq 6$.
Moreover, by Lemma~\ref{0101-3}, we have $k \neq 2$,
and hence $k=4$.
According to the computer search in~\cite{vDS},
the only connected $4$-regular graph with $7$ vertices and four distinct eigenvalues is the complement of $C_7$,
which is isomorphic to the circulant graph $\Cay(\MB{Z}_7 , \{\pm 1, \pm 2\})$.
This graph has the eigenvalue $\lambda = \zeta_7 + \zeta_7^{-1} + \zeta_7^2 + \zeta_7^{-2} = -1 -\zeta_7^3 - \zeta_7^4$,
and hence, by our assumption,
$\Omega \ni \frac{2 \lambda}{4} = \frac{1}{2}(-1 -\zeta_7^3 - \zeta_7^4) \in \MB{Q}(\zeta_7)$.
However, the condition $\frac{1}{2}(-1 -\zeta_7^3 - \zeta_7^4) \in \Omega \cap \MB{Q}(\zeta_7)$ contradicts Lemma~\ref{0206-2}.
\end{proof}

\begin{lem} \label{0115-1}
Let $\G$ be a genuine strongly $\ell$-walk-regular graph with distinct eigenvalues $k, \theta_1, \theta_2, \theta_3$.
If $\frac{2\theta_i}{k}$ is an algebraic integer for every $i \in \{1,2,3\}$, then the eigenvalues $\theta_1, \theta_2, \theta_3$ must be integers.
\end{lem}
\begin{proof}
By Lemma~\ref{0109-2},
it suffices to show that case~(ii) of Lemma~\ref{0109-1} does not occur.
Suppose, for contradiction, that $\G$ has eigenvalues $\theta_1 = \theta \in \MB{Z}$, $\theta_2 = a+b\sqrt{c}$ and $\theta_3 = a-b\sqrt{c}$,
where $a$ and $b$ are rational numbers,
and $c$ is a square-free integer.
We may assume that $b > 0$ without loss of generality.

First, we consider the case where $c \not\equiv 1 \pmod{4}$.
From our assumption and Lemma~\ref{0129-1},
we have $\frac{2}{k}(a \pm b \sqrt{c}) \in \Omega \cap \MB{Q}(\sqrt{c}) = \{ p+q\sqrt{c} \mid p, q \in \MB{Z} \}$.
Thus, there exist integers $p$ and $q$ such that
$\frac{2}{k}(a \pm b \sqrt{c}) = p \pm q \sqrt{c}$.
Moreover, we have $q > 0$ because $b > 0$.
Since $-2 \leq \frac{2}{k}(a \pm b \sqrt{c}) < 2$,
we have
\begin{align}
-2 < p + q \sqrt{c} < 2, \label{0109-4} \\
-2 \leq p - q \sqrt{c} < 2. \label{0109-5}
\end{align}
Eliminate $p$ using Inequalities~\eqref{0109-4} and~\eqref{0109-5}.
In addition, from $q > 0$ it follows that $0 < q < \frac{2}{\sqrt{c}} \leq \frac{2}{\sqrt{2}} < 2$,
and hence $q=1$.
Eliminate $\sqrt{c}$ using Inequalities~\eqref{0109-4} and~\eqref{0109-5} again.
Then,
we have $-2 < p < 2$, i.e., $p = 0, \pm 1$.
However, $(p,q) = (\pm 1, 1)$ does not satisfy either~\eqref{0109-4} or~\eqref{0109-5} because $c \geq 2$.
Therefore, we have $(p,q) = (0, 1)$,
i.e., $(\theta_2, \theta_3) = (\frac{k}{2}\sqrt{c}, -\frac{k}{2}\sqrt{c})$.
By Corollary~\ref{0113-3},
the remaining integral eigenvalue must satisfy $\theta = 0$.
However, these eigenvalues do not sum to zero,
contradicting Lemma~\ref{0127-1}.

Next, we consider the case where $c \equiv 1 \pmod{4}$.
From our assumption and Lemma~\ref{0129-1},
we have $\frac{2}{k}(a \pm b \sqrt{c}) \in \Omega \cap \MB{Q}(\sqrt{c}) = \{ p+\frac{1+\sqrt{c}}{2}q \mid p, q \in \MB{Z} \}$.
Thus, there exist integers $p$ and $q$ such that $\frac{2}{k}(a + b \sqrt{c}) = (p+\frac{q}{2}) + \frac{q}{2}\sqrt{c}$,
and hence $\frac{2}{k}(a - b \sqrt{c}) = (p+\frac{q}{2}) - \frac{q}{2}\sqrt{c}$.
Moreover, we have $q > 0$ because $b > 0$.
Since $-2 \leq \frac{2}{k}(a \pm b \sqrt{c}) < 2$,
we have
\begin{align}
-2 < \left(p+\frac{q}{2} \right) + \frac{q}{2}\sqrt{c} < 2, \label{0113-1} \\
-2 \leq \left(p+\frac{q}{2} \right) - \frac{q}{2}\sqrt{c} < 2. \label{0113-2}
\end{align}
Eliminate $p$ using Inequalities~\eqref{0113-1} and~\eqref{0113-2}.
In addition, from $q > 0$ it follows that $0 < q < \frac{4}{\sqrt{c}} \leq \frac{4}{\sqrt{5}} < 2$,
and hence $q=1$.
%
Eliminate $\sqrt{c}$ using Inequalities~\eqref{0113-1} and~\eqref{0113-2} again.
Then,
we have $-4 < 2p+1 < 4$, i.e., $p \in \{-2, -1, 0, 1\}$.
However, $(p,q) = (-2, 1), (1, 1)$ does not satisfy either~\eqref{0113-1} or~\eqref{0113-2} because $c \geq 5$.
To complete the proof,
we will reject both cases $(p,q) = (0,1), (-1,1)$.

We consider the case where $(p,q) = (0,1)$.
Inequality~\eqref{0113-1} implies that $c = 5$.
Thus, the two irrational eigenvalues are determined to be $a \pm b \sqrt{5} = \frac{k}{4}(1 \pm \sqrt{5})$.
From Lemma~\ref{0403-1},
the possible values for the remaining integer eigenvalue $\theta$ are $\pm \frac{k}{2}, 0$, or $-k$.
If $\theta = -k$, then the graph $\G$ is bipartite.
However, its spectrum is not symmetric with respect to $0$.
If $\theta = 0$, then it contradicts Corollary~\ref{0113-3}.
If $\theta = \frac{k}{2}$, then the eigenvalues do not sum to zero.
We would like to reject the case $\theta = -\frac{k}{2}$.
Let $\Spec(\G) = \{ [k]^1, [-\frac{k}{2}]^{\alpha}, [\frac{k}{4}(1 \pm \sqrt{5})]^{\beta} \}$.
By Lemma~\ref{0127-1},
we have
\begin{align*}
n &= 1 + \alpha + 2\beta, \\
0 &= 
k - \frac{k}{2}\alpha
+ \left\{ \frac{k}{4}(1 + \sqrt{5}) \right\}\beta
+ \left\{ \frac{k}{4}(1 - \sqrt{5}) \right\} \beta, \\
nk &= k^2 + \frac{k^2}{4} \alpha
+ \left\{ \frac{k}{4}(1 + \sqrt{5}) \right\}^2 \beta
+ \left\{ \frac{k}{4}(1 - \sqrt{5}) \right\}^2 \beta.
\end{align*}
Eliminating $n$ and $\alpha$, and reformulating, yields $(k-3)(4\beta + 3) = 3$.
This implies that $(k, \beta) = (4, 0), (6, -1/2)$.
Both cases contradict the condition $\beta > 0$.

Finally, we consider the case where $(p,q) = (-1,1)$.
Inequality~\eqref{0113-2} implies that $c = 5$.
The two irrational eigenvalues are determined to be $a \pm b \sqrt{5} = \frac{k}{4}(-1 \pm \sqrt{5})$.
From~Lemma~\ref{0403-1},
the possible values for the remaining integer eigenvalue $\theta$ are $\pm \frac{k}{2}, 0$, or $-k$.
If $\theta = -k$, then the graph $\G$ is bipartite.
However, its spectrum is not symmetric with respect to $0$.
If $\theta = 0$, then it contradicts Corollary~\ref{0113-3}.
If $\theta = -\frac{k}{2}$, 
then setting $\Spec(\G) = \{ [k]^1, [-\frac{k}{2}]^{\alpha}, [\frac{k}{4}(-1 \pm \sqrt{5})]^{\beta} \}$ gives $\alpha + \beta = 2$ as the sum of the eigenvalues is zero.
This implies that $\alpha = \beta = 1$,
and hence $|V(\G)| = 4$.
The only connected regular graphs with $4$ vertices are $K_4$ and $C_4$, but neither of these graphs is appropriate in this case.
We would like to reject the case $\theta = \frac{k}{2}$.
Let $\Spec(\G) = \{ [k]^1, [\frac{k}{2}]^{\alpha}, [\frac{k}{4}(-1 \pm \sqrt{5})]^{\beta} \}$.
By Lemma~\ref{0127-1},
we have
\begin{align*}
n &= 1 + \alpha + 2\beta, \\
0 &= 
k + \frac{k}{2}\alpha
+ \left\{ \frac{k}{4}(-1 + \sqrt{5}) \right\}\beta
+ \left\{ \frac{k}{4}(-1 - \sqrt{5}) \right\} \beta, \\
nk &= k^2 + \frac{k^2}{4} \alpha
+ \left\{ \frac{k}{4}(-1 + \sqrt{5}) \right\}^2 \beta
+ \left\{ \frac{k}{4}(-1 - \sqrt{5}) \right\}^2 \beta.
\end{align*}
Eliminating $n$ and $\alpha$, and reformulating, yields $(k-3)(2\beta + 1) = -5$.
In particular, we have $k - 3 < 0$, i.e., $k=2$.
This contradicts Lemma~\ref{0101-3}.
\end{proof}

\begin{thm} \label{0122-1}
Let $\G$ be a genuine strongly $\ell$-walk-regular graph with distinct eigenvalues $k > \theta_1 > \theta_2 > \theta_3$.
If $\frac{2\theta_i}{k}$ is an algebraic integer for every $i \in \{1,2,3\}$, then $(\theta_1, \theta_2, \theta_3) = (\frac{k}{2}, 0, -\frac{k}{2})$. 
\end{thm}
\begin{proof}
By Lemma~\ref{0115-1},
we have $\theta_1, \theta_2, \theta_3 \in \MB{Z}$.
By Lemma~\ref{0403-1},
$\theta_i \in \{\pm \frac{k}{2}, 0, -k\}$.
If $\theta_3 = -k$,
then $\G$ is bipartite,
and hence $\theta_1 = \frac{k}{2}$ and $\theta_2 = -\frac{k}{2}$.
However, $(\theta_1, \theta_2, \theta_3) = (\frac{k}{2}, -\frac{k}{2}, -k)$ does not satisfy Equality~\eqref{0103-1} for any $\ell \geq 3$.
Thus, the eigenvalues must satisfy $(\theta_1, \theta_2, \theta_3) = (\frac{k}{2}, 0, -\frac{k}{2})$.
\end{proof}

\begin{cor}
Let $\G$ be a genuine strongly walk-regular graph with distinct eigenvalues $k > \theta_1 > \theta_2 > \theta_3$.
If $\G$ admits perfect state transfer, then $(\theta_1, \theta_2, \theta_3) = (\frac{k}{2}, 0, -\frac{k}{2})$. 
\end{cor}

\begin{proof}
Let $P = P(\G)$.
By Corollary~\ref{0404-1}, we have $\Theta_P(e_x) = \sigma(P)$ for any $x \in V(\G)$.
Since $\G$ admits perfect state transfer,
it follows from Lemma~\ref{0128-5} that $\frac{2\theta_i}{k}$ is an algebraic integer for every $i \in \{1,2,3\}$.
Therefore, by Theorem~\ref{0122-1}, we conclude that $(\theta_1, \theta_2, \theta_3) = (\frac{k}{2}, 0, -\frac{k}{2})$.
\end{proof}

\begin{cor} \label{0214-1}
Let $\G$ be a genuine strongly walk-regular graph with distinct eigenvalues $k > \theta_1 > \theta_2 > \theta_3$.
Then, $\G$ is periodic if and only if $(\theta_1, \theta_2, \theta_3) = (\frac{k}{2}, 0, -\frac{k}{2})$. 
\end{cor}
\begin{proof}
If $\G$ is periodic,
it follows from~\cite[Lemma~3.2]{K} that $\frac{2\theta_i}{k}$ is an algebraic integer for every $i \in \{1,2,3\}$.
By Theorem~\ref{0122-1},
we have $(\theta_1, \theta_2, \theta_3) = (\frac{k}{2}, 0, -\frac{k}{2})$.
Conversely, if the eigenvalues satisfy $(\theta_1, \theta_2, \theta_3) = (\frac{k}{2}, 0, -\frac{k}{2})$,
then the spectral mapping theorem (see, for example, \cite[Theorem~2.4]{K}) yields $\sigma(U) = \{ \pm1, e^{\pm \frac{\pi}{2}i},
e^{\pm \frac{\pi}{3}i},
e^{\pm \frac{2\pi}{3}i}\}$,
and hence $U^{12} = I$.
Therefore, we see that the graph $\G$ is periodic.
\end{proof}

We note that a very general form of the spectral mapping theorem is given in~\cite{SS}.
However, we cited~\cite{K} in the previous proof for convenience in setting up the notation.

\section{Regular graphs with spectrum $\{ [k], [\pm \frac{k}{2}], [0] \}$} \label{S5}

In the previous section,
we saw that if a genuine strongly walk-regular graph admits perfect state transfer or is periodic,
then its spectrum must be of the form $\{[k]^1, [\frac{k}{2}]^{\alpha}, [0]^{\beta}, [-\frac{k}{2}]^{\gamma}\}$.
In this section, we provide more precise analysis of the multiplicities of spectra of this form.
We note that regular graphs with four distinct eigenvalues are walk-regular~\cite[Section~3]{vD}, that is,
if $A$ is the adjacency matrix of such a graph,
then the quantity $(A^r)_{x,x}$ is independent of the vertex $x$ and depends only on a positive integer $r$.

\begin{thm} \label{0206-1}
Let $\G = (V,E)$ be a 
$k$-regular graph with spectrum $\Spec(\Gamma) = \{[k]^1, [\frac{k}{2}]^{\alpha}, [0]^{\beta}, [-\frac{k}{2}]^{\gamma} \}$,
and let $n=|V|$.
Then, the following statements hold:
\begin{enumerate}[(i)]
\item $\alpha = \frac{2n}{k} - 3, \beta = n+3 - \frac{4n}{k}$, and $\gamma = \frac{2n}{k} -1$.
In particular, these are positive integers.
\item $2k \leq n \leq \frac{3}{4}k^3$.
\item Let $t_x$ be the number of triangles passing through vertex $x \in V$.
Then, we have $t_x = \frac{3k^3}{8n}$.
In particular, this is an integer.
%
\item Let $q$ be the number of quadrangles in $\G$,
and let $q_x$ be the number of quadrangles passing through vertex $x \in V$.
Then,
\begin{align*}
q &= \frac{k}{32} \left( 3k^3 + nk^2 - 8nk + 4n \right ), \\
q_x &= \frac{k}{8n} \left( 3k^3 + nk^2 - 8nk + 4n \right ).
\end{align*}
In particular, these are integers.
\end{enumerate}
\end{thm}

\begin{proof}
(i) has already been proved in the end of the proof of Lemma~\ref{0124-1}.

(ii) For the lower bound, it follows from $1 \leq \alpha = \frac{2n}{k} - 3$.
We show the upper bound.
Let $A$ be the adjacency matrix of $\G$.
Since $\sum_{\lambda \in \Spec(\G)} \lambda^3 = \frac{3}{4}k^3 > 0$,
we have $(A^3)_{x,x} \geq 1$ for any vertex $x$.
Let $h(x)$ be the Hoffman polynomial of $\G$,
that is, $h(x) := (x- \frac{k}{2})x(x + \frac{k}{2})$.
Then we have $h(A) = \frac{h(k)}{n}J$~\cite[Theorem~3.3.2]{BH},
i.e.,
\[ A^3 = \frac{k^2}{4}A + \frac{h(k)}{n}J. \]
Thus, we have $1 \leq (A^3)_{x,x} = (\frac{k^2}{4}A + \frac{h(k)}{n}J)_{x,x} = \frac{3k^3}{4n}$,
and hence $n \leq \frac{3}{4}k^3$.

(iii) We recall that $\G$ is walk-regular.
Moreover,
Since $(A^3)_{x,x}$ counts the number of closed walks of length $3$ starting from $x$ and ending at $x$,
we have $(A^3)_{x,x} = 2t_x$.
Thus,
\begin{align*}
2nt_x = \sum_{x \in V}(A^3)_{x,x}
= \tr(A^3) = \sum_{\lambda \in \Spec(\G)} \lambda^3 = \frac{3k^3}{4},
\end{align*}
and hence $t_x = \frac{3k^3}{8n}$.

(iv) The number of quadrangles in a regular graph can be counted using eigenvalues (see~\cite[Lemma~5.5]{K} for details).
We have 
\[ n(2k^2 - k) + 8q = \sum_{\lambda \in \Spec(\G)} \lambda^4 =
\frac{k^4}{4} \left( \frac{n}{k} + 3 \right). \]
Solving this for $q$ yields $q = \frac{k}{32}( 3k^3 + nk^2 - 8nk + 4n)$.
Since $\G$ is walk-regular, we have
\[ q_x = \frac{4q}{n} = \frac{k}{8n} \left( 3k^3 + nk^2 - 8nk + 4n \right ). \]
\end{proof}

By Theorem~\ref{0206-1},
we can narrow down the candidates for graphs whose spectrum is of the form $\{[k]^1, [\frac{k}{2}]^{\alpha}, [0]^{\beta}, [-\frac{k}{2}]^{\gamma} \}$.
Indeed, condition (ii) makes the number of vertices $n$ finite for a fixed $k$,
and conditions (i), (iii), and (iv) eliminate unsuitable values of $n$.
We call graphs to pass Theorem~\ref{0206-1} \emph{feasible graphs}.

Tables~\ref{t1} and \ref{t2} list feasible graphs for $k \leq 20$.
In these tables, there are many graphs whose existence is unknown.
The column ``Existence" provides one example that realizes the spectrum, if such graphs exist.
The graph $\G_{C([n,k], (w_1, w_2, w_3))^{\perp}}$ represents the coset graph of the dual of a projective binary three-weight $[n,k]$ code with weights $(w_1,w_2,w_3)$.
For more details on this graph, see~\cite{KKSSW}.
In addition, the symbol ``$-$" in the tables denotes that the spectrum passes Theorem~\ref{0206-1} but the non-existence of a graph is shown by other reasons.
The reasons are briefly described in the rightmost column of the tables.
For a graph $\G$ and a positive integer $m$,
we denote by $\G \otimes J_m$
the graph whose adjacency matrix is $A(\G) \otimes J_m$.
The graph $\G \otimes J_m$ is called the \emph{$m$-coclique extension} of $\G$,
which is obtained by replacing each vertex of $\G$ with an $m$-coclique.

As mentioned in the proof of Theorem~\ref{0206-1}~(iii),
$(A^r)_{x,x}$ is independent of vertex $x$ and depends only on a positive integer $r$.
Let $c_r$ denote this value. 
Then,
\[ c_r = \frac{1}{n} \sum_{\lambda \in \Spec(A)} \lambda^r, \]
and in particular, the right-hand side must be a non-negative integer.
However, imposing this condition for $r$ up to $20$, in addition to the four conditions of Theorem~\ref{0206-1} did not refine the tables.
Also, Nozaki established an upper bound on the number of vertices of connected regular graphs by using the number of irreducible closed walks i.e., walks that do not contain any backtracking $x \sim y \sim x$~\cite[Theorem~2]{N}.  
In our attempt, however, replacing the count of irreducible closed walks with that of closed walks did not lead to any improvement in the tables.

\begin{table}[H]
\centering
\caption{Feasible graphs for $k \in \{4, 6, \dots, 14\}$} \label{t1}
\begin{tabular}{c|c|c|c|c} \hline
$k$ & $n$ & Spectrum & Existence & Comment \\ \hline
4 & 8 &  $\{ [4]^{1}, [2]^{1}, [0]^{3}, [-2]^{3} \}$ & $\overline{H(3,2)}$ & \\ 
4 & 12 &  $\{ [4]^{1}, [2]^{3}, [0]^{3}, [-2]^{5} \}$ & $L(H(3,2))$ & \\ 
\hline 
6 & 27 &  $\{ [6]^{1}, [3]^{6}, [0]^{12}, [-3]^{8} \}$ & $H(3,3)$ & \\ 
6 & 81 &  $\{ [6]^{1}, [3]^{24}, [0]^{30}, [-3]^{26} \}$ & ? & $q=0$ \\ 
\hline 
8 & 16 &  $\{ [8]^{1}, [4]^{1}, [0]^{11}, [-4]^{3} \}$ & $\overline{H(3,2)} \otimes J_2$ & \\ 
8 & 24 &  $\{ [8]^{1}, [4]^{3}, [0]^{15}, [-4]^{5} \}$ & $L(H(3,2)) \otimes J_2$ & \\ 
8 & 32 &  $\{ [8]^{1}, [4]^{5}, [0]^{19}, [-4]^{7} \}$ & $\G_{C([8,5],(2,4,6))^{\perp}}$ & \cite[Table~1]{KKSSW} \\ 
8 & 48 &  $\{ [8]^{1}, [4]^{9}, [0]^{27}, [-4]^{11} \}$ & ? & \\ 
8 & 64 &  $\{ [8]^{1}, [4]^{13}, [0]^{35}, [-4]^{15} \}$ & $\G_{C([8,6],(2,4,6))^{\perp}}$ & \cite[Table~1]{KKSSW} \\ 
8 & 96 &  $\{ [8]^{1}, [4]^{21}, [0]^{51}, [-4]^{23} \}$ & ? & \\ 
8 & 192 &  $\{ [8]^{1}, [4]^{45}, [0]^{99}, [-4]^{47} \}$ & ? & \\ 
\hline 
10 & 25 &  $\{ [10]^{1}, [5]^{2}, [0]^{18}, [-5]^{4} \}$ & $-$ & \cite[Proposition~8]{vDS} \\ 
10 & 75 &  $\{ [10]^{1}, [5]^{12}, [0]^{48}, [-5]^{14} \}$ & ? & \\ 
10 & 125 &  $\{ [10]^{1}, [5]^{22}, [0]^{78}, [-5]^{24} \}$ & ? & \\ 
10 & 375 &  $\{ [10]^{1}, [5]^{72}, [0]^{228}, [-5]^{74} \}$ & ? & \\ 
\hline 
12 & 24 &  $\{ [12]^{1}, [6]^{1}, [0]^{19}, [-6]^{3} \}$ & $\overline{H(3,2)} \otimes J_3$ & \\ 
12 & 36 &  $\{ [12]^{1}, [6]^{3}, [0]^{27}, [-6]^{5} \}$ & $L(H(3,2)) \otimes J_3$ & \\ 
12 & 54 &  $\{ [12]^{1}, [6]^{6}, [0]^{39}, [-6]^{8} \}$ & $H(3,3) \otimes J_2$ & \\ 
12 & 72 &  $\{ [12]^{1}, [6]^{9}, [0]^{51}, [-6]^{11} \}$ & ? & \\ 
12 & 108 &  $\{ [12]^{1}, [6]^{15}, [0]^{75}, [-6]^{17} \}$ & ? & \\ 
12 & 162 &  $\{ [12]^{1}, [6]^{24}, [0]^{111}, [-6]^{26} \}$ & ? & \\ 
12 & 216 &  $\{ [12]^{1}, [6]^{33}, [0]^{147}, [-6]^{35} \}$ & ? & \\ 
12 & 324 &  $\{ [12]^{1}, [6]^{51}, [0]^{219}, [-6]^{53} \}$ & ? & \\ 
12 & 648 &  $\{ [12]^{1}, [6]^{105}, [0]^{435}, [-6]^{107} \}$ & ? & \\ 
\hline 
14 & 49 &  $\{ [14]^{1}, [7]^{4}, [0]^{38}, [-7]^{6} \}$ & ? & \\ 
14 & 147 &  $\{ [14]^{1}, [7]^{18}, [0]^{108}, [-7]^{20} \}$ & ? & \\ 
14 & 343 &  $\{ [14]^{1}, [7]^{46}, [0]^{248}, [-7]^{48} \}$ & ? & \\ 
14 & 1029 &  $\{ [14]^{1}, [7]^{144}, [0]^{738}, [-7]^{146} \}$ & ? &
\\ \hline 
\end{tabular}
\end{table}

\begin{table}[H]
\centering
\caption{Feasible graphs for $k \in \{16, 18, 20\}$} \label{t2}
\begin{tabular}{c|c|c|c|c} \hline
$k$ & $n$ & Spectrum & Existence & Comment \\ \hline
16 & 32 &  $\{ [16]^{1}, [8]^{1}, [0]^{27}, [-8]^{3} \}$ & $\overline{H(3,2)} \otimes J_4$ & \\ 
16 & 48 &  $\{ [16]^{1}, [8]^{3}, [0]^{39}, [-8]^{5} \}$ & $L(H(3,2)) \otimes J_4$ & \\ 
16 & 64 &  $\{ [16]^{1}, [8]^{5}, [0]^{51}, [-8]^{7} \}$ & $\G_{C([8,5],(2,4,6))^{\perp}} \otimes J_2$ & \\ 
16 & 96 &  $\{ [16]^{1}, [8]^{9}, [0]^{75}, [-8]^{11} \}$ & ? & \\ 
16 & 128 &  $\{ [16]^{1}, [8]^{13}, [0]^{99}, [-8]^{15} \}$ & $\G_{C([8,6],(2,4,6))^{\perp}} \otimes J_2$ & \\ 
16 & 192 &  $\{ [16]^{1}, [8]^{21}, [0]^{147}, [-8]^{23} \}$ & ? & \\ 
16 & 256 &  $\{ [16]^{1}, [8]^{29}, [0]^{195}, [-8]^{31} \}$ & ? & \\ 
16 & 384 &  $\{ [16]^{1}, [8]^{45}, [0]^{291}, [-8]^{47} \}$ & ? & \\ 
16 & 512 &  $\{ [16]^{1}, [8]^{61}, [0]^{387}, [-8]^{63} \}$ & ? & \\ 
16 & 768 &  $\{ [16]^{1}, [8]^{93}, [0]^{579}, [-8]^{95} \}$ & ? & \\ 
16 & 1536 &  $\{ [16]^{1}, [8]^{189}, [0]^{1155}, [-8]^{191} \}$ & ? & \\ 
\hline
18 & 81 &  $\{ [18]^{1}, [9]^{6}, [0]^{66}, [-9]^{8} \}$ & $H(3,3) \otimes J_3$ & \\ 
18 & 243 &  $\{ [18]^{1}, [9]^{24}, [0]^{192}, [-9]^{26} \}$ & ? & \\ 
18 & 729 &  $\{ [18]^{1}, [9]^{78}, [0]^{570}, [-9]^{80} \}$ & ? & \\ 
18 & 2187 &  $\{ [18]^{1}, [9]^{240}, [0]^{1704}, [-9]^{242} \}$ & ? & \\ 
\hline 
20 & 40 &  $\{ [20]^{1}, [10]^{1}, [0]^{35}, [-10]^{3} \}$ & $\overline{H(3,2)} \otimes J_5$ & \\ 
20 & 50 &  $\{ [20]^{1}, [10]^{2}, [0]^{43}, [-10]^{4} \}$ & ? & \\ 
20 & 60 &  $\{ [20]^{1}, [10]^{3}, [0]^{51}, [-10]^{5} \}$ & $L(H(3,2)) \otimes J_5$ & \\ 
20 & 100 &  $\{ [20]^{1}, [10]^{7}, [0]^{83}, [-10]^{9} \}$ & ? & \\ 
20 & 120 &  $\{ [20]^{1}, [10]^{9}, [0]^{99}, [-10]^{11} \}$ & ? & \\ 
20 & 150 &  $\{ [20]^{1}, [10]^{12}, [0]^{123}, [-10]^{14} \}$ & ? & \\ 
20 & 200 &  $\{ [20]^{1}, [10]^{17}, [0]^{163}, [-10]^{19} \}$ & ? & \\ 
20 & 250 &  $\{ [20]^{1}, [10]^{22}, [0]^{203}, [-10]^{24} \}$ & ? & \\ 
20 & 300 &  $\{ [20]^{1}, [10]^{27}, [0]^{243}, [-10]^{29} \}$ & ? & \\ 
20 & 500 &  $\{ [20]^{1}, [10]^{47}, [0]^{403}, [-10]^{49} \}$ & ? & \\ 
20 & 600 &  $\{ [20]^{1}, [10]^{57}, [0]^{483}, [-10]^{59} \}$ & ? & \\ 
20 & 750 &  $\{ [20]^{1}, [10]^{72}, [0]^{603}, [-10]^{74} \}$ & ? & \\ 
20 & 1000 &  $\{ [20]^{1}, [10]^{97}, [0]^{803}, [-10]^{99} \}$ & ? & \\ 
20 & 1500 &  $\{ [20]^{1}, [10]^{147}, [0]^{1203}, [-10]^{149} \}$ & ? & \\ 
20 & 3000 &  $\{ [20]^{1}, [10]^{297}, [0]^{2403}, [-10]^{299} \}$ & ? & \\ 
\hline 
\end{tabular}
\end{table}

\section{Further discussion}

From Corollary~\ref{0214-1},
all graphs in Tables~\ref{t1} and~\ref{t2} are periodic.
However, we do not yet know whether these graphs admit perfect state transfer.
In this section, we investigate this question and conclude that none of the graphs appearing in the ``Existence" column of the tables admit perfect state transfer.

First, we show that $m$-coclique extensions for $m \geq 3$ do not admit perfect state transfer.
Roughly speaking,
this can be shown by investigating symmetry of the graphs.
We will use the following result.

\begin{lem}[{\cite[Corollary~3.4]{KY}}] \label{0214-2}
Let $\G = (V, E)$ be a graph, and let $x,y \in V$.
If perfect state transfer occurs from $d^*e_x$ to $d^*e_y$ on $\G$,
then the automorphism groups that fix x and y coincide, i.e.,
$\Aut(\G)_x = \Aut(\G)_y$.
\end{lem}

Let $[n] = \{1,2, \dots, n\}$ for a positive integer $n$.
The vertex set and edge set of the $m$-coclique extension $\G \otimes J_m$ of a graph $\G$ can be expressed as 
\begin{align*}
V(\G \otimes J_m) &= V(\G) \times [m], \\
E(\G \otimes J_m) &= \{ \{(x,i), (y,j)\} \mid \{x,y\} \in E(\G), i,j \in [m] \},
\end{align*}
respectively.
Let $S_m$ be the symmetric group on the set $[m]$.
For a vertex $z$ of $\G$ and a permutation $\sigma \in S_m$,
the mapping $\varphi_{z,\sigma}: V(\G \otimes J_m) \to V(\G \otimes J_m)$ defined by 
\begin{align*}
    \varphi_{z,\sigma}(x, i) = \begin{cases}
        (x, \sigma(i)) &\text{if } x = z,\\
        (x,i) &\text{otherwise.}
    \end{cases}
\end{align*}
is an automorphism of $\G \otimes J_m$.

\begin{thm} \label{0322-1}
Let $\G$ be a graph.
Then $\G \otimes J_m$ does not admit perfect state transfer for any integer $m \geq 3$.
\end{thm}


\begin{proof}
Suppose that the graph $\G \otimes J_m$ admits perfect state transfer from $d^*e_{(x,i)}$ to $d^*e_{(y, j)}$.
Since $m \geq 3$, there exists $k \in [m] \setminus \{i,j\}$.
Then, the automorphism $\varphi_{y, (jk)}$
fixes $(x,i)$ but not $(y, j)$.
This contradicts the equality $\Aut(\G \otimes J_m)_{(x,i)} = \Aut(\G \otimes J_m)_{(y,j)}$ in Lemma~\ref{0214-2}.
\end{proof}

This theorem does not hold for $m = 2$ because $C_4 = P_2 \otimes J_2$ is a counterexample.

Next, we discuss the minimal time at which perfect state transfer occurs.
The following proposition enables us to verify by computer whether the graphs appearing in Tables~\ref{t1} and~\ref{t2} admit perfect state transfer.

\begin{pro} \label{0322-2}
Let $\G = (V, E)$ be a connected $k$-regular graph whose eigenvalues are given by $\sigma(A(\G)) = \{k ,\pm \frac{k}{2}, 0\}$.
If $\G$ admits perfect state transfer,
its minimal time is either $6$ or $12$.
\end{pro}
\begin{proof}
We assume that $\G$ admits perfect state transfer from $d^*e_x$ to $d^* e_y$ for $x,y \in V$.
By Corollary~\ref{0109-6}, $\G$ is genuine strongly walk-regular,
and hence 
it follows from Corollary~\ref{0404-1} that $\Theta_A(e_x) = \sigma(A)$.
Let $P$ be the discriminant of $\G$.
By Theorem~\ref{pst},
we have $E_{\lambda}e_x = \pm E_{\lambda}e_y$ for any $\lambda \in \sigma(P)$, and thus we define
\begin{align*}
\Theta_P^+ &= \{ \lambda \in \Theta_P(e_x) \mid E_{\lambda}e_x = E_{\lambda}e_y \}, \\
\Theta_P^- &= \{ \lambda \in \Theta_P(e_x) \mid E_{\lambda}e_x = - E_{\lambda}e_y \}.
\end{align*}
The conditions (b) and (c) of (C) in Theorem~\ref{pst} imply that,
since $(1, \frac{1}{2}, -\frac{1}{2}, 0)
= (\cos \frac{0}{6}\pi, \cos \frac{2}{6}\pi, \cos \frac{4}{6}\pi, \cos \frac{3}{6}\pi)$,
the times at which perfect state transfer can occur are multiples of $6$,
and we have $1, \pm \frac{1}{2} \in \Theta_P^+$.
If $0 \in \Theta_P^+$, we have
\[ \left( 1, \frac{1}{2}, -\frac{1}{2}, 0 \right)
= \left(\cos \frac{0}{12}\pi, \cos \frac{4}{12}\pi, \cos \frac{8}{12}\pi, \cos \frac{6}{12}\pi \right), \]
and hence perfect state transfer occurs at time $12$.
If $0 \in \Theta_P^-$, we have
\[ \left(1, \frac{1}{2}, -\frac{1}{2}, 0 \right)
= \left(\cos \frac{0}{6}\pi, \cos \frac{2}{6}\pi, \cos \frac{4}{6}\pi, \cos \frac{3}{6}\pi \right), \]
and hence perfect state transfer occurs at time $6$.
In particular, the minimum time at which perfect state transfer occur is either $6$ or $12$.
\end{proof}

Using Theorem~\ref{0322-1}, Proposition~\ref{0322-2} and a computer,
it is confirmed that none of the graphs appearing in the ``Existence" column of Tables~\ref{t1} and~\ref{t2}
admit perfect state transfer.
At present, it is unknown whether there exist any graphs with eigenvalues $\{k ,\pm \frac{k}{2}, 0\}$ that admit perfect state transfer.

\begin{question}
Does there exist a connected $k$-regular graph with eigenvalues $\{k ,\pm \frac{k}{2}, 0\}$ that admits perfect state transfer?
If not, prove that no such graph exists.
\end{question}

\section*{Acknowledgements}
We would like to thank Mr. Fang Muh Yeong for providing the key idea for the proof of Lemma~\ref{0416-1}.
This paper is based on the master's thesis of the third author, Harunobu Yata,
and we would like to thank his supervisor at the time, Professor Norio Konno.
This work is supported by JSPS KAKENHI (Grant Number JP24K16970).

\section*{Data Availability Statement}
This study is purely theoretical and does not involve any data generation or analysis.

\end{document}